\documentclass[12pt]{amsart}

\usepackage{amssymb}
\usepackage{bm}
\usepackage{graphicx}
\usepackage{psfrag}
\usepackage{color}
\usepackage{enumerate}
\usepackage{url}

\usepackage{algpseudocode}

\usepackage{mathtools}

\usepackage{xy}
\input xy
\xyoption{all}

\newcommand{\excise}[1]{}

\newtheorem{thm}{Theorem}[section]
\newtheorem{lemma}[thm]{Lemma}

\newtheorem{cor}[thm]{Corollary}
\newtheorem{prop}[thm]{Proposition}

\theoremstyle{definition}

\newtheorem{example}[thm]{Example}
\newtheorem{remark}[thm]{Remark}
\newtheorem{defn}[thm]{Definition}

\newtheorem{notation}[thm]{Notation}
\newtheorem{question}[thm]{Question}
\newtheorem{prob}[thm]{Problem}

\numberwithin{equation}{section}

\newcommand{\ring}[1]{\ensuremath{\mathbb{#1}}}

\renewcommand\>{\rangle}
\newcommand\<{\langle}

\newcommand\ZZ{\ring{Z}}

\newcounter{case}

\begin{document}

\mbox{}
\title[On arithmetical numerical monoids with some generators omitted]{On arithmetical numerical monoids \\ with some generators omitted}

\author{Sung Hyup Lee}
\address{Mathematics Department\\University of California Berkeley\\Berkeley, CA 94720}
\email{sunghlee@berkeley.edu}

\author{Christopher O'Neill}
\address{Mathematics Department\\University of California Davis\\Davis, CA 95616}
\email{coneill@math.ucdavis.edu}

\author{Brandon Van Over}
\address{Mathematics Department\\University of California Berkeley\\Berkeley, CA 94720}
\email{bvanover@berkeley.edu}

\date{\today}

\begin{abstract}
Numerical monoids (cofinite, additive submonoids of the non-negative integers) arise frequently in additive combinatorics, and have recently been studied in the context of factorization theory.  
Arithmetical numerical monoids, which are minimally generated by arithmetic sequences, are particularly well-behaved, admitting closed forms for many invariants that are difficult to compute in the general case.  
In~this paper, we answer the question ``when does omitting generators from an arithmetical numerical monoid $S$ preserve its (well-understood) set of length sets and/or Frobenius number?''\ in two extremal cases: (i) we characterize which individual generators can be omitted from $S$ without changing the set of length sets or Frobenius number; and (ii) we prove that under certain conditions, nearly every generator of $S$ can be omitted without changing its set of length sets or Frobenius number.  
\end{abstract}

\maketitle


\section{Introduction}
\label{sec:intro}

\emph{Numerical monoids} (additively closed subsets of the non-negative integers) were first studied in the context of the Frobenius coin-exchange problem, which asks for the smallest integer that cannot be evenly changed using a given collection of relatively prime coin values.  More recently, numerical monoids have been examined in the realm of factorization theory, which aims to classify and quantify the different ways monoid elements can be expressed as sums of generators (called \emph{factorizations}).  

One of the central objects of study in factorization theory is the set $\mathcal L(S)$ of length sets of a given monoid $S$ (Definition~\ref{d:numerical}), which encodes a large amount of information about the factorization structure of $S$.  Several combinatorially-flavored factorization invariants, such as elasticity and the delta set \cite{nonuniq}, are derived directly from the set of length sets, and the Narkiewicz conjecture \cite{narkiewicz}, one of the largest open problems in factorization theory, is deeply rooted in the use of set of length sets as a monoid isomorphism invariant.  That said, there are very few monoids whose set of length sets is completely understood \cite{lensetprogress}.  

\emph{Arithmetical} numerical monoids, whose minimal generating sets are arithmetic sequences, are considered to be one of the most well-behaved families of numerical monoids.  Indeed, many monoid-theoretic invariants that are difficult to obtain for a general numerical monoid $S$ admit concise closed formulas in terms of $a$, $w$ and $d$ for an arithmetical numerical monoid $S = \<a, a + d, \ldots, a + wd\>$.  A classical example is the Frobenius number $F(S)$ (taking its name from the aforementioned coin-exchange problem), defined as the largest integer that lies outside of $S$ \cite{diophantinefrob}.  Closed formulas are also known for the genus, Ap\'ery set, delta set, and catenary degree of an arithmetical numerical monoid \cite{catenaryarithseq,omidali,numerical}.  Arithmetical numerical monoids are also one of the only families of numerical monoids (or monoids in general) whose set of length sets is completely parametrized \cite{setoflengthsets}.  

In this paper, we examine numerical monoids obtained by omitting generators from an arithmetical numerical monoid $S = \<a, a + d, \ldots, a + wd\>$.  Intuitively, if a generator $a + rd$ that is ``deep in the middle of the list'' is omitted, then the factorization structure remains largly unchanged, as in any factorization of some $n \in S$, we can substitute
$$2(a + rd) = (a + (r-1)d) + (a + (r+1)d),$$
or more generally,
$$(a + rd) + (a + r'd) = (a + (r-1)d) + (a + (r'+1)d).$$
These concise minimal relations are one of the reasons arithmetical numerical monoids are easier to work with than more general numerical monoids \cite{omidali}.  

Our main question is as follows.  Note that removing the first or last generator of $S$ could instead be accomplished by simply modifying $a$ and $w$ appropriately.  

\begin{question}\label{q:mainquestion}
Fix an arithmetical numerical monoid $S = \<a, a + d, \ldots, a + wd\>$ with $w < a$ and $\gcd(a,d) = 1$, fix a subset $G \subset \{a + d, \ldots, a + (w-1)d\}$ of the generators of $S$, and let 
$$S' = \<a + rd : 0 \le r \le w \textnormal{ and } a + rd \notin G\>.$$
\begin{enumerate}[(a)]
\item 
\label{q:mainquestion:len}
Under what conditions does $\mathcal L(S) = \mathcal L(S')$?  

\item 
\label{q:mainquestion:frob}
Under what conditions does $F(S) = F(S')$?  

\end{enumerate}
\end{question}

Our results are threefold.  First, we give an algorithm for determining whether $\mathcal L(S) = \mathcal L(S')$ in Question~\ref{q:mainquestion}\eqref{q:mainquestion:len}.  The algorithm, and its implementation, are discussed in Remark~\ref{r:algorithm}.  Second, we prove the rather surprising fact that if $a$ is sufficiently large, then the monoid $S_* = \<a, a + d, a + (w-1)d, a + wd\>$ has set of length sets and Frobenius number identical to $S$, as does every numerical monoid in between (Corollaries~\ref{c:missingalllen} and~\ref{c:missingallfrob}).  These two results fit into the recent literature concerning ``shifted'' numerical monoids; see Remark~\ref{r:shifted}.  Lastly, we completely answer both parts of Question~\ref{q:mainquestion} when precisely one generator is omitted from $S$ (Corollary~\ref{c:miss1genlen} and Theorem~\ref{t:miss1genfrob}).

\subsection*{Acknowledgements}

The authors would like to thank 
Vadim Ponomarenko for many helpful conversations.

\subsection*{Dedication}

The third author would like to dedicate this paper to his late father, Thomas Van Over, who taught him the perseverence necessary to complete this paper.

%




\section{Background}
\label{sec:background}

\begin{defn}\label{d:numerical}
A \emph{numerical monoid} $S$ is a cofinite, additive submonoid of $\ZZ_{\ge 0}$.  When we write $S = \<n_1, \ldots, n_w\>$, we assume $n_1 < \cdots < n_w$ is the unique minimal generating set of $S$ (with respect to containment).  
A \emph{factorization} of $n \in S$ is an expression 
$$n = z_1n_1 + \cdots + z_wn_w$$
of $n$ as a sum of generators of $S$, and the \emph{length} of a factorization is the number $z_1 + \cdots + z_w$.  
of generators appearing the sum.  The \emph{length set} of $n$ is the set 
$$\mathsf L_S(n) = \{z_1 + \cdots + z_w : n = z_1n_1 + \cdots + z_wn_w, z_i \in \ZZ_{\ge 0}\},$$
of all possible factorization lengths of $n$, and the \emph{set of length sets} of $S$ is denoted 
$$\mathcal L(S) = \{\mathsf L_S(n) : n \in S\}.$$
The \emph{maximum} and \emph{minimum} factorization length functions are defined as 
$$\mathsf M_S(n) = \max \mathsf L_S(n) \qquad \text{ and } \qquad \mathsf m_S(n) = \min \mathsf L_S(n),$$
respectively.  A numerical monoid $S$ is \emph{arithmetical} if $S = \<a, a + d, \ldots, a + wd\>$ with $\gcd(a,d) = 1$ and $w < a$ (these conditions ensure $S$ is cofinite and minimally generated).  
\end{defn}

Central to many of the arguments in this paper is the following membership criterion for arithmetical numerical monoids.  

\begin{lemma}[{\cite[Theorem~3.1]{omidali}, \cite[Theorem~2.2]{setoflengthsets}}]\label{l:omidali}
Fix $c_1, c_2 \in \ZZ$, and let $n = c_1a + c_2d$. 
\begin{enumerate}[(a)]
\item 
\label{l:omidali:membership}
If $0 \le c_2 < a$, then $n \in S$ if and only if $c_2 \le c_1w$.  

\item 
\label{l:omidali:len}
More generally, $c_1 \in \mathsf L_S(n)$ if and only if $c_2 \le c_1w$.  In particular, $\mathsf M_S(n) = c_1$ and
$$|\mathsf L_S(n)| = \displaystyle \bigg\lfloor \frac{c_1w - c_2}{a + wd} \bigg\rfloor + 1.$$

\end{enumerate}
\end{lemma}

\section{An algorithm for equality of set of length sets}
\label{sec:algorithm}

\begin{notation}\label{n:sstar}
In this section, fix $a, d, w \in \ZZ_{\ge 1}$ with $\gcd(a, d) = 1$ and $4 \le w < a$.  Write $n_i = a + id$ for $0 \le i \le w$, and let 
$$S = \<n_0, \ldots, n_w\> \qquad \text{ and } \qquad S_* = \<n_0, n_1, n_{w-1}, n_w\>.$$
\end{notation}

The main result of this section is Theorem~\ref{t:algorithm}, which states that $L_S(n)$ and $L_{S_*}(n)$ agree for $n$ sufficiently large.  This has several immediate consequences.  First, this yields an algorithm to determine when a monoid $S'$ obtained by omitting generators of $S$ satisfies $\mathcal L(S) = \mathcal L(S')$ (Remark~\ref{r:algorithm}).  Additionally, Theorem~\ref{t:algorithm} implies that when $a$~is sufficiently large, $\mathcal L(S) = \mathcal L(S_*)$ and $F(S) = F(S_*)$ (Corollaries~\ref{c:missingalllen} and~\ref{c:missingallfrob}), which effectively gives an upper bound on when Question~\ref{q:mainquestion} has a nontrivial answer.  Corollaries~\ref{c:missingalllen} and~\ref{c:missingallfrob} closely resemble recent results for ``shifted'' numerical monoids; see Remark~\ref{r:shifted} for more detail.  

\begin{thm}\label{t:algorithm}
Fix $c_1, c_2 \in \ZZ$ with $0 \le c_2 < a$, and let $n = c_1a + c_2d$.  Whenever $n > (w - 3)(a + wd)$, the following are equivalent: (i)~$n \in S_*$; (ii)~$n \in S$; and (iii)~$L_S(n) = L_{S_*}(n)$.  
\end{thm}

\begin{proof}
Suppose $n \in S$.  It suffices to prove $L_S(n) \subset L_{S_*}(n)$, as the other implications are immediate.  Fix $\ell \in \mathsf L_S(n)$, and as in the proof of Theorem~\ref{t:miss1genlen}, write $n - a\ell = pd$ and $p = qw + r$ for $p,q,r \in \ZZ$ with $0 \le r < w$, and consider the length $\ell$ factorization
$$n = (\ell - 1 - q)a + (a + rd) + q(a + wd).$$
This gives $\ell \in \mathsf L_{S_*}(n)$ unless $1 < r < w - 1$.  
By assumption, $\mathsf m_S(n) \ge w - 2$, so either
$$\text{(i) } q \ge w - r - 1
\qquad \text{ or } \qquad
\text{(ii) }\ell - q \ge r,$$
as otherwise $\ell = q + (\ell - q) < w - 2$.  If (i) holds, then we can write 
$$(a + rd) + q(a + wd) = (w - r)(a + (w-1)d) + (q - (w-r-1))(a + wd),$$
and if (ii) holds, then
$$(\ell - 1 - q)a + (a + rd) = (\ell - q - r)a + r(a + d).$$
In either case, we conclude $n \in S_*$ and $\ell \in L_{S_*}(n)$.
\end{proof}

\begin{remark}\label{r:algorithm}
Given a set $G \subset \{2, \ldots, w - 2\}$, the proof of Theorem~\ref{t:algorithm} yields an algorithm for determining whether $\mathcal L(S') = \mathcal L(S)$ for $S' = \<a + id : 0 \le i \le w, i \notin G\>$.  In particular, $\mathcal L(S') = \mathcal L(S)$ if and only if 
$$\{\mathsf L_S(n) : n \in S, \mathsf m_S(n) \le w - 3\} = \{\mathsf L_{S'}(n) : n \in S', \mathsf m_{S'}(n) \le w - 3\},$$
both of which only contain finitely many length sets, each of which can be readily computed from the (finite and computable) set of factorizations.  

The primary computational hurdle is computing the length sets of all elements $n \le (w - 3)(a + wd)$.  Thankfully, this can be done relatively quickly using dynamic programming \cite[Algorithm~3.7]{dynamicalg}.  In particular, this algorithm avoids computing the set of factorizations of each $n$, which has on the order of $n^{w-2}$ elements, and instead computes only the length set of each $n$, whose cardinality is linear in $n$.  

We will see in Corollary~\ref{c:miss1genlen} that omitting $a + d$ or $a + (w-1)d$ results in a monoid with distinct set of length sets from $S$, so the requirement that $1, w-1 \notin G$ in the above algorithm has little impact on its use investigating Question~\ref{q:mainquestion}.  
\end{remark}

All experimental evidence generated using the algorithm in Remark~\ref{r:algorithm} supports an affirmitive answer to the following problem.  

\begin{prob}\label{pr:simplicial}
Is the collection $\mathcal G = \mathcal P(\{1, \ldots, (w-1)\})$ of all sets $G$ satisfying
$$\mathcal L(\<a + id : 0 \le i \le w, i \notin G\>) = \mathcal L(S)$$
closed under taking subsets?  In particular, does $G' \subset G$ and $G \in \mathcal G$ imply $G' \in \mathcal G$?  
\end{prob}

\begin{example}\label{e:algorithm}
In cases where Problem~\ref{pr:simplicial} has a positive answer, the collection $\mathcal G$ can be recorded in a relatively concise manner, as it suffices to list either (i) only the elements of $\mathcal G$ that are maximal with respect to containment, or (ii) only the minimal sets that lie outside of $\mathcal G$.  This is effectively viewing $\mathcal G$ as an abstract simplicial complex; see \cite[Chapter~2]{stanleycca} for more background on this.  

We include several examples, each computed using a \texttt{Sage} \cite{sage} implementation of the algorithm in Remark~\ref{r:algorithm} that internally uses the \texttt{GAP} package \texttt{numericalsgps} \cite{numericalsgpsgap}.  
\begin{enumerate}[(a)]
\item 
If $S = \<11, 12, \ldots, 18\>$, then $\{3,4,5\}$ is the only set outside of $\mathcal G \subset \mathcal P(\{2, \ldots, 5\})$.  

\item 
If $S = \<23, 26, \ldots, 56\>$, then $\mathcal G \subset \mathcal P(\{2, \ldots, 9\})$ has 7 maximal sets, but
$$\{2, 3, 4, 5, 6, 7, 8\} \qquad \text{ and } \qquad \{3, 4, 5, 6, 7, 8, 9\}$$
are the only minimal sets outside of $\mathcal G$.  

\item 
If $S = \<51, 53, \ldots, 67\>$, then $\mathcal G = \mathcal P(\{2, \ldots, 6\})$, a consequence of Corollary~\ref{c:missingalllen}.  

\end{enumerate}
\end{example}

\begin{cor}\label{c:missingalllen}
If $a \ge w^2 - 3w$, then $\mathcal L(S) = \mathcal L(S')$ for any monoid $S'$ obtained from~$S$ by omitting a subset of the generators $a + 2d, \ldots, a + (w-2)d$.  
\end{cor}

\begin{proof}
By Theorem~\ref{t:algorithm}, it suffice to consider $n \le (w - 3)(a + wd)$.  First, we claim $|L_S(n)| \le 1$.  Indeed, write $n =  c_1a + c_2d$ for $c_1, c_2 \in \ZZ$ and $0 \le c_2 < a$.  
Since
$$c_1 = \mathsf M_S(n) \le \frac{(w - 3)(a + wd)}{a}$$
and $a \ge w^2 - 3w$, we have
\begin{align*}
\frac{c_1w - c_2}{a + wd} \le \frac{c_1w}{a + wd} \le \frac{(w - 3)w}{a} \le 1.
\end{align*}
The claim now follows from Lemma~\ref{l:omidali}\eqref{l:omidali:len}.  

At this point, all we have left to show is $\mathsf L_S(n) \in \mathcal L(S_*)$.  By Lemma~\ref{l:omidali}\eqref{l:omidali:len}, 
$$\mathsf L_{S_*}(c_1a) \subset \mathsf L_S(c_1a) \subset \mathsf L_S(c_1a + c_2d) = \mathsf L_S(n),$$
all of which equal $\{c_1\}$ by the above claim.  This completes the proof.  
\end{proof}

\begin{cor}\label{c:missingallfrob}
If $a > w^2 - 3w + 1$, then $F(S) = F(S')$ for any monoid $S'$ obtained from~$S$ by omitting a subset of the generators $a + 2d, \ldots, a + (w-2)d$.  
\end{cor}

\begin{proof}
It suffices to show $F(S) = F(S_*)$.  If $a > w^2 - 3w + 1$, then by \cite[Theorem~3.3.2]{diophantinefrob}, 
$$F(S) = \bigg\lceil \frac{a - 1}{w} \bigg\rceil a - a + (a - 1)d \ge (w - 3)a + (a - 1)d > (w - 3)(a + wd),$$
after which $S$ and $S_*$ agree by Theorem~\ref{t:algorithm}.  As such, $F(S) = F(S_*)$.  
\end{proof}

\begin{remark}\label{r:missingallfrob}
Based on experiments, the bound $a > w^2 - 3w + 1$ given in Corollary~\ref{c:missingallfrob} is near tight for every $w \ge 6$.  This was tested using the \texttt{GAP} package \texttt{numericalsgps}~\cite{numericalsgpsgap} for $6 \le w \le 25$ and $d \le 10$; in well over half of the cases, the largest value of~$a$ for which $F(S) \ne F(S_*)$ was exactly $w^2 - 3w + 1$, and in every test case run, it~was at most 3 less than this value.  
\end{remark}

\begin{remark}\label{r:shifted}
There has been a recent interest in ``shifted'' numerical monoids, namely those of the form $\<n, n + r_1, \ldots, n + r_w\>$.  Asymptotic behavior (i.e.\ for $n > r_w^2$) of the delta set and Betti numbers \cite{shiftydelta,shiftyminpres} (both of which are determined by the set of length sets) and the Frobenius number \cite{shiftyapery} have been characterized.  Corollaries~\ref{c:missingalllen} and~\ref{c:missingallfrob} can be seen as special cases of these results, but with a key advantage:\ our lower bounds on $a$ do not depend on $d$.  This is particularly surprising when one considers how critical the size of $r_w$ is when working with general shifted numerical monoids.  Among other things, this yields a significantly stronger bound than those given in the aforementioned papers whenever $d \ge 2$.  
\end{remark}

\section{Omitting a single generator}
\label{sec:miss1gen}

\begin{notation}\label{n:sr}
Throughout this section, fix $a, d, w, r \in \ZZ_{\ge 1}$ with $\gcd(a, d) = 1$ and $1 \leq r \leq (w-1)$.  Write $n_i = a + id$ for $0 \le i \le w$, and let 
$$S = \<n_0, \ldots, n_w\> \qquad \text{ and } \qquad S_r = \<n_0, \dots, \widehat{n_r}, \dots, n_w\>.$$
In particular, $S_r$ is the monoid obtained from $S$ by omitting~$n_r$ as a generator.  
\end{notation}

In this section, we completely answer Question~\ref{q:mainquestion} for the monoids $S_r$ (Corollary~\ref{c:miss1genlen} and Theorem~\ref{t:miss1genfrob}).  In the process, we expand Lemma~\ref{l:omidali}\eqref{l:omidali:membership} to give a membership criterion for $S_r$ (Proposition~\ref{p:membership}) and we completely parametrize the set of length sets of $S_r$ (Theorem~\ref{t:miss1genlen}).  

\begin{prop}\label{p:membership}
Fix $c_1, c_2 \in \ZZ$ with $0 \le c_2 < a$, and let $n = c_1a + c_2d$.  
\begin{enumerate}[(a)]
\item 
\label{p:membership:middle}
If $1 < r < w - 1$, then $S \setminus S_r = \{n_r\}$.

\item 
\label{p:membership:firstgen}
$n \in S \setminus S_1$ if and only if $c_2 = 1$ and $wc_1 \le n_w$.  

\item 
\label{p:membership:lastgen}
$n \in S \setminus S_{w-1}$ if and only if $c_2 = c_1w - 1$ and $wc_1 \le n_0$.  

\end{enumerate}
\end{prop}

\begin{proof}
For part~(a), notice that (i)~$2n_r = n_{r-1} + n_{r+1}$, (ii)~$n_r + n_{r+1} = n_{r-1} + n_{r+2}$, (iii)~$n_{r-1} + n_r = n_{r-2} + n_{r+1}$, and for any $k \ge 2$, (iv)~$n_r + n_{r+k} = n_{r+1} + n_{r+k-1}$ and (v)~$n_r + n_{r-k} = n_{r-1} + n_{r-k+1}$.  As such, $S \setminus S_r = \{n_r\}$.  

For part~(b), a similar argument shows that $n_1 + n_k \in S_1$ for any $k \ge 1$.  This leaves elements of the form $n_1 + kn_0 = ka + d$, which by Lemma~\ref{l:omidali}\eqref{l:omidali:len} lie in $S_1$ whenever we can write $ka + d = (k-d)a + (a+1)d$ so that both $k-d \ge 0$ and $(k-d)w \ge a+1$.  The latter inequality is stictly more restrictive, so part~(b) is proved.  

Part~(c) follows by an analogous argument to part~(b), so the proof is complete.  
\end{proof}

\begin{thm}\label{t:miss1genlen}
Fix $c_1, c_2 \in \ZZ$ with $0 \le c_2 < a$, let $n = c_1a + c_2d$, and let 
$$K = \frac{c_1w - c_2}{a + wd}.$$
\begin{enumerate}[(a)]
\item 
If $1 < r < w - 1$, then $\mathsf L_{S_r}(n) = \mathsf L_S(n) = \{c_1 - kd : 0 \le k \le K\}$.  

\item 
$\ell \in \mathsf L_S(n) \setminus \mathsf L_{S_1}(n)$ if and only if $n \equiv d \bmod a$ and $\ell = \max \mathsf L_S(n)$.  

\item 
$\ell \in \mathsf L_S(n) \setminus \mathsf L_{S_{w-1}}(n)$ if and only if $n \equiv -d \bmod (a + wd)$ and $\ell = \min \mathsf L_S(n)$.  

\end{enumerate}
\end{thm}

\begin{proof}
By Lemma~\ref{l:omidali}\eqref{l:omidali:len} and the definition of $S_r$, 
$$\mathsf L_{S_r}(n) \subset \mathsf L_S(n) = \{c_1 - kd : 0 \le k \le K\},$$
and in particular each $\ell \in \mathsf L_{S_r}(n)$ can be written as $\ell = c_1 - kd$ for some $k \in \ZZ_{\ge 0}$.  
In~order to bound $k$, consider the inequalities
$$\ell a \le n \le \ell (a + wd).$$
Combined with the fact that $n \equiv a\ell \bmod d$, we obtain $n - a\ell = pd$ for some $p \in \ZZ$.  
If~$n = (a + wd)\ell$, then this is the only factorization of $n$ in both $S$ and $S_r$.  Otherwise, writing $p = qw + r_0$ for $0 \leq r_0 < w$, we see that
$$n = \ell a + pd = (\ell - 1 - q)a + (a + r_0d) + q(a + wd),$$
is a factorization of $n$ of length $\ell$ in $S$, since 
$$0 \le n - \ell a < \ell (a + wd) - \ell a = \ell wd,$$
implies $0 \leq q \leq \ell - 1$.  
From here, it suffices to consider the case $r_0 = r$.  
If $r > 1$ and $q < \ell - 1$, we may adjust the factorization for $n$ above to obtain 
$$n = (\ell - 2 - q)a + (a + d) + (a + (r-1)d) + q(a + wd)$$
is a factorization of length $\ell$ in $S_r$.  On the other hand, if $r < w - 1$ and $q > 0$, then
$$n = (\ell - 1 - q)a + (a + (r+1)d) + (a + (w-1)d) + (q - 1)(a + wd)$$
is such a factorization.  This proves part~(a) and leaves the two cases below, comprising exactly the conditions for which $\ell \notin \mathsf L_{S_r}(n)$ specified in parts~(b) and~(c), respectively.  

First, suppose $r = 1$ and $q = 0$.  Any length $\ell$ factorization $n = h_0a + \cdots + h_w(a + wd)$ of $n$ in $S$ can be rearranged to give
$$\ell a + d = n = \ell a + d(h_1 + 2h_2 + \cdots),$$
which is only possible if $h_1 > 0$.  Moreover, this is precisely the case $n = (a + d) + (\ell - 1)a$, which is equivalent to $n \equiv d \mod a$ and $\ell = \max \mathsf L_S(n)$, thus proving part~(b).  

Second, suppose $r = w - 1$ and $q = \ell - 1$.  
Analogous to the first case, any length $\ell$ factorization $n = h_0a + \cdots + h_w(a + wd)$ of $n$ in $S$ can be rearranged as
$$\ell (a + wd) - d = n = \sum_{i = 0}^w h_{w-i}(a + (w-i)d) = \ell (a + wd) - d\bigg(\sum_{i = 0}^w h_{w-i}i\bigg)$$
which is only possible if $h_{w-1} = 1$.  
Additionally, this is precisely the case in which $n = (a + (w-1)d) + q(a + wd)$, which is equivalent to $n \equiv -d \bmod (a + wd)$ and $\ell = \min \mathsf L_S(n)$.  Hence, part~(c) is verified, and the proof is complete.
\end{proof}

\begin{cor}\label{c:miss1genlen}
$\mathcal L(S_r) = \mathcal L(S)$ if and only if $1 < r < w - 1$.
\end{cor}

\begin{proof}
If $1 < r < w - 1$, then $\mathcal L(S_r) = \mathcal L(S)$ by Theorem~\ref{t:miss1genlen}.  Conversely, fix $c_1 \in \ZZ_{\ge 0}$ with $n = c_1a + d \in S_1$, and suppose $c_1w = t(a + wd)$ for $t \in \ZZ$.  We claim $L_{S_1}(n) \notin \mathcal L(S)$.  

By Theorem~\ref{t:miss1genlen}, $\mathsf M_{S_1}(n) = c_1 - d$ and $|\mathsf L_{S_1}(n)| = t - 1$.  Fix $n' = c_1'a + c_2'd \in S$ with $0 \le c_2' < a$ so that $\mathsf M_S(n') = \mathsf M_{S_1}(n)$.  Again by Theorem~\ref{t:miss1genlen}, $\mathsf M_S(n') = c_1' = c_1 - d$ and 
$$|\mathsf L_S(n')| = \bigg\lfloor \frac{(c_1-d)w - c_2'}{a + wd} \bigg\rfloor + 1
= \bigg\lfloor \frac{-(c_2' + wd)}{a + wd} \bigg\rfloor  + t + 1 \ge t.$$
In particular, $|\mathsf L_S(n')| > |\mathsf L_{S_1}(n)|$, so the proof is complete.
\end{proof}

It turns out $\mathcal L(S_1)$ and $\mathcal L(S_{w-1})$, though distinct from $\mathcal L(S)$, are equal to one another.  Our proof of this curious fact uses Lemma~\ref{l:lastgenequiv}, which gives an equivalent version of the modular equation in Theorem~\ref{t:miss1genlen}(c).  

\begin{lemma}\label{l:lastgenequiv}
For $n = c_1a + c_2d$ with $0 \le c_2 < a$, the following are equivalent:
\begin{enumerate}[(a)]
\item 
$n \equiv -d \bmod (a + wd)$; and

\item 
$c_1w \equiv (c_2 + 1) \bmod (a + wd)$.  

\end{enumerate}
\end{lemma}

\begin{proof}
Suppose $c_1w - (c_2 + 1) = m(a + wd)$ for $m \in \ZZ$.  Then 
$$n + d = c_2d + c_1a + d = c_2d - c_1wd + d + c_1(a + wd) = (md + c_1)(a + wd).$$
Conversely, suppose $n + d = M(a + wd)$ for $M \in \ZZ$.  Since $\gcd(a,d) = 1$, reducing both expressions for $n$ modulo $d$ yields $M \equiv c_1 \bmod d$, so writing $M = md + c_1$ for $m \in \ZZ$,
$$c_2d + d - c_1wd = c_2d + d + c_1a - c_1(a + wd) = n + d - c_1(a + wd) = md(a + wd),$$
at which point dividing by $d$ completes the proof.  
\end{proof}

\begin{thm}\label{t:boundaries}
$\mathcal L(S_1) = \mathcal L(S_{w-1})$.
\end{thm}

\begin{proof}
By Theorem~\ref{t:miss1genlen}, equality of length sets can be verified by simply comparing their cardinality and maximum elements, and it suffices to verify $\mathsf L_{S_1}(n) \in \mathcal L(S_{w-1})$ and $\mathsf L_{S_{w-1}}(n) \in \mathcal L(S_1)$ for elements $n = c_1a + c_2d$ with $0 \le c_2 < a$ satisfying either $n \equiv d \bmod a$ or $n \equiv -d \bmod (a + wd)$.  Throughout the proof, we use the notation
$$[m] := \lfloor m/(a + wd) \rfloor.$$

We begin by showing $\mathcal L(S_{w-1}) \subset \mathcal L(S_1)$.  For $n = c_1a + d \in S_{w-1}$, Theorem~\ref{t:miss1genlen}(a) ensures $\mathsf M_{S_{w-1}}(n) = \mathsf M_{S_1}(n + d) = \mathsf M_{S_1}(n - d) = c_1$, and by Lemma~\ref{l:lastgenequiv}, 
$$|\mathsf L_{S_1}(n + d)|
= [c_1w - 2] + 1
= [c_1w - 1] + 1
= |\mathsf L_{S_{w-1}}(n)|$$
whenever $c_1w \not\equiv 1 \bmod (a + wd)$, and otherwise, 
$$|\mathsf L_{S_1}(n - d)|
= [c_1w] + 1
= [c_1w - 1] + 1
= |\mathsf L_{S_{w-1}}(n)|.$$
Next, suppose $n \equiv -d \bmod (a + wd)$.  We claim that $L_{S_{w-1}}(n)$ agrees with either $L_{S_1}(n + 2d)$ or $L_{S_1}(n + 3d)$.  We first show that $n + 2d, n + 3d \in S_1$.  By Proposition~\ref{p:membership}\eqref{p:membership:lastgen}, $wc_1 > a$, which implies $w(c_1 + d) > a + wd$.  This ensures 
$$n + 2d = (c_1 + d)a + d \qquad \text{ and } \qquad n + 3d = (c_1 + d)a + 2d$$
both lie in $S_1$ by Lemma~\ref{l:omidali}\eqref{l:omidali:membership} and Proposition~\ref{p:membership}\eqref{p:membership:firstgen}.  
Now, Lemma~\ref{l:lastgenequiv} implies $c_1w = m(a + wd) + (c_2 + 1)$ for $m \in \ZZ$.  By Theorem~\ref{t:miss1genlen}, 
$$\mathsf M_{S_{w-1}}(n) = c_1 \qquad \text{ and } \qquad |\mathsf L_{S_{w-1}}(n)|
= [c_1w - c_2]
= m.$$
If $c_2 = a - 1$, then $\mathsf M_{S_1}(n + 2d) = c_1 = \mathsf M_{S_{w-1}}(n)$ and 
$$|\mathsf L_{S_1}(n + 2d)|
= [(c_1 + d)w - 1]
= [m(a + wd) + a + wd - 1]
= m.$$
If $c_2 = a - 2$, then by similar reasoning $\mathsf L_{S_1}(n + 3d) = \mathsf L_{S_{w-1}}(n)$.  In all remaining cases, $\mathsf M_{S_1}(n + 2d) = c_1$ and
$$|\mathsf L_{S_1}(n + 2d)|
= [c_1w - (c_2 + 2)] + 1
= m + [-2] + 1
= m,$$
which completes the proof that $\mathcal L(S_{w-1}) \subset \mathcal L(S_1)$.  

It remains to show $\mathcal L(S_1) \subset \mathcal L(S_{w-1})$.  
Suppose $n \equiv -d \bmod (a + wd) \in S_1$, so that $(a + wd) \mid c_1w - c_2 - 1$ by Lemma~\ref{l:lastgenequiv}.  Based on whether or not $c_2 = a - 1$, either $n' = n + d$ or $n' = n - d$ will satisfy $\mathsf M_{S_{w-1}}(n') = c_1 = \mathsf M_{S_1}(n)$ and
$$|\mathsf L_{S_{w-1}}(n')|
= [c_1w - c_2 \pm 1] + 1
= [c_1w - c_2] + 1
= |\mathsf L_{S_1}(n)|.$$

Next, suppose $n = c_1a + d$.  We claim that $L_{S_1}(n)$ agrees with either $L_{S_{w-1}}(n - 2d)$ or $L_{S_{w-1}}(n - 3d)$.  We first show that $n - 2d, n - 3d \in S_{w-1}$.  By Proposition~\ref{p:membership}\eqref{p:membership:firstgen}, $c_1w > a + wd$, so in particular $c_1 > d$ and $a - 1 \le (c_1 - d)w$.  This means 
$$n - 3d = (c_1 - d)a + (a - 2)d
\qquad \text{ and } \qquad
n - 2d = (c_1 - d)a + (a - 1)d$$
both lie in $S_{w-1}$ by Lemma~\ref{l:omidali}\eqref{l:omidali:membership} and Proposition~\ref{p:membership}\eqref{p:membership:lastgen}.  
By Theorem~\ref{t:miss1genlen}, 
$$\mathsf M_{S_1}(n) = c_1 - d
\qquad \text{ and } \qquad
|\mathsf L_{S_1}(n)| = [c_1w - 1].$$
By Theorem~\ref{t:miss1genlen}, $\mathsf M_S(-)$ and $\mathsf M_{S_{w-1}}(-)$ coincide on elements of $S_{w-1}$, which means
$\mathsf M_{S_{w-1}}(n - 2d) = \mathsf M_{S_{w-1}}(n - 3d) = c_1 - d$.  
If we have $c_1w \equiv -1 \bmod (a + wd)$, then $n - 3d \equiv -d \bmod (a + wd)$, and if $c_1w \equiv 0 \bmod (a + wd)$, then $n - 2d \equiv -d \bmod (a + wd)$, and in either case we are done by the above argument.  In all cases that remain, $n - 2d \not\equiv -d \bmod (a + wd)$ by Lemma~\ref{l:lastgenequiv}, so 
$$
|\mathsf L_{S_{w-1}}(n - 2d)|
= [(c_1 - d)w - (a - 1)]
= [c_1w + 1] - 1
= |\mathsf L_{S_1}(n)|
$$
by Theorem~\ref{t:miss1genlen}\eqref{p:membership:lastgen}, which completes the proof.  
\end{proof}

\begin{thm}\label{t:miss1genfrob}
The following hold true.  
\begin{enumerate}[(a)]
\item 
If $1 < r < w - 1$, then $F(S_r) = F(S)$.  

\item 
$F(S_1) > F(S)$.  

\item 
$F(S_{w-1}) > F(S)$ if and only if either (i) $w \mid a$, or (ii) $w \mid a - 1$ and $d < a$.  

\end{enumerate}
\end{thm}

\begin{proof}
By \cite[Theorem~3.3.2]{diophantinefrob}, $S$ has Frobenius number 
$$F(S) = \bigg(\bigg\lceil \frac{a-1}{w} \bigg\rceil - 1\bigg)a + (a - 1)d.$$
Part~(a) follows immediately from Proposition~\ref{p:membership}\eqref{p:membership:middle}, and for part~(b),
$$F(S) < \bigg(\bigg\lceil \frac{a-1}{w} \bigg\rceil - 1 + d\bigg)a < \bigg(\bigg\lfloor \frac{a}{w} \bigg\rfloor + d\bigg)a + d,$$
which lies outside of $S_1$ by Proposition~\ref{p:membership}\eqref{p:membership:firstgen}.  Lastly, any $n = c_1a + (c_1w - 1)d \notin S_{w-1}$ satisfies $c_1w \le a$ by Proposition~\ref{p:membership}\eqref{p:membership:lastgen}, so
$$n = c_1a + (c_1w - 1)d \le \lfloor a/w \rfloor a + (a - 1)d = F(S)$$
as long as $w \nmid a, a - 1$.  If $w \mid a$, then  
$$F(S) < \bigg\lceil \frac{a-1}{w} \bigg\rceil a + (a - 1)d = \bigg(\frac{a}{w}\bigg)a + (a - 1)d,$$
which lies outside of $S_{w-1}$ by Proposition~\ref{p:membership}\eqref{p:membership:lastgen}.  In the final case, 
$$F(S) = \bigg(\bigg\lceil \frac{a-1}{w} \bigg\rceil - 1\bigg)a + (a - 1)d
= \bigg\lfloor \frac{a}{w} \bigg\rfloor a + \bigg(\bigg\lfloor \frac{a}{w} \bigg\rfloor w - 1\bigg)d + (d - a) = n + d - a,$$
where $n$ is the largest element of $S \setminus S_{w-1}$.  Part~(c) immediately follows.  
\end{proof}


\end{document}